\documentclass[11pt, article]{amsart}
\usepackage[pagewise]{lineno}

\usepackage{amsfonts}
\usepackage{amsmath}
\usepackage{amsthm}
\usepackage{amssymb}
\usepackage[utf8]{inputenc} 
\usepackage[T1]{fontenc}    
\usepackage{hyperref}       
\usepackage{url}            
\usepackage{booktabs}       
\usepackage{amsfonts}       
\usepackage{nicefrac}       
\usepackage{microtype}      
\usepackage{lipsum}
\usepackage[ruled,vlined]{algorithm2e}
\usepackage{algpseudocode}
\usepackage{subfig}
\usepackage{siunitx}
\usepackage[pdftex]{graphicx}
\usepackage[all]{xy}
\usepackage[customcolors]{hf-tikz}
\usepackage{nicematrix}
\usepackage{xcolor,lipsum}

\def\FF{\mathbb{F}}

\def\O{\mathcal{O}}
\def\f{\varphi}
\def\di{\partial}

\DeclareMathOperator{\argmin}{argmin}
\DeclareMathOperator{\rank}{rank}
\DeclareMathOperator{\diam}{diam}
\DeclareMathOperator{\row}{Row}
\DeclareMathOperator{\low}{Low}
\DeclareMathOperator{\col}{Col}
\DeclareMathOperator{\Rips}{Rips}
\DeclareMathOperator{\Ker}{Ker}
\DeclareMathOperator{\im}{Im}

\newcommand{\ii}[1]{\tikzmarkin[set fill color=blue!10,set border color=blue]{#1}(-.3,-.2)(.3,.4){}1 \tikzmarkend{#1}}

\newtheorem{theorem}{Theorem}
\newtheorem{remark}{Remark}

\newtheorem{proposition}{Proposition}

\title[Fast computation of persistent homology representatives...]{Fast computation of persistent homology representatives with involuted persistent homology}

\author{Matija \v Cufar}
\address{NZ Institute for Advanced Study, Massey University, Auckland, New Zealand}
\email{matijacufar@gmail.com}

\author{\v Ziga Virk}
\address{University of Ljubljana and Institute IMFM, Ljubljana, Slovenia}
\email{ziga.virk@fri.uni-lj.si}

\thanks{The second author was  supported by Slovenian Research Agency grants No. N1-0114, J1-4001, J1-4031, and P1-0292.
The authors would like to thank the referee for careful reading and valuable comments.}
\begin{document}
\maketitle

\begin{abstract}
Persistent homology is typically computed through persistent cohomology. While this generally improves the running time significantly, it does not facilitate extraction of homology representatives.
The mentioned representatives are geometric manifestations of the corresponding holes and often carry desirable information.

 We propose a new method of extraction of persistent homology representatives using cohomology. In summary, we first compute persistent cohomology and use the obtained information to significantly improve the running time of the direct persistent homology computations. This algorithm applied to Rips filtrations generally computes persistent homology representatives much faster than the standard methods.
\end{abstract}

\section{Introduction}
Persistent homology~\cite{EdelsHarer} is a widely applicable stable descriptor of metric spaces. It is usually computed from coboundary matrices, in effect computing the isomorphic persistent cohomology. The reason is that the structure of cohomology is more amenable to certain speedups and in most cases results in significantly faster computation~\cite{bauer2019ripser, de2011circular, de2011dualities} than the original computation of persistent homology from the boundary matrices~\cite{ELZ}.

While persistent homology contains information about the lifespans of homology elements interpreted as holes in various dimensions, the actual cycles representing these holes can only be extracted from the reduced boundary matrix. When computations are carried out by reducing the coboundary matrix though, we can only obtain representative cocycles. These are favorable in some specific settings~\cite{de2011circular}, but in general cycles are the much preferred method of visualization and expression of persistent homology elements. The difference between them is demonstrated in Figure~\ref{fig:cycles-cocycles}.

\textbf{Contributions}. In this paper, we present an algorithm that allows us to leverage the speed of
persistent cohomology computation and to recover representative cycles. The essential part of the algorithm has two distinct phases:
\begin{enumerate}
 \item Reduce the coboundary matrix and extract simplices contributing to the computation of representative cycles.
 \item Reduce the boundary matrix restricted to the columns determined by the previous part.
\end{enumerate}

In effect we use the reduced coboundary matrix to determine the death simplices of persistence pairs and then ``ignore'' all other columns in the reduction of the boundary matrix.
We call this approach involuted homology computation.
The efficiency of our algorithm as compared to the standard approach depends on the disparity of the reduction times for persistent homology and cohomology.
While the algorithm works for any filtration, we restrict our comparison to Rips complexes, as the speedup there is the most significant.
A comparative analysis suggests our approach is the fastest way to obtain the representatives in almost all settings.

\textbf{Related work}. In this paper we describe how to compute homology representatives arising from the reduction process of persistent homology in any dimension. The standard way of obtaining the representatives through the reduced boundary matrix is given in Subsection~\ref{SubsRepresentative}: the representatives are the columns corresponding to the death simplices in the reduced boundary matrix. We apply the idea of this procedure in phase 2 to a much restricted boundary matrix. Phase 1 is based on speedups to the reduction of the coboundary matrix achieved by Ripser~\cite{bauer2019ripser}. The speedups arise from certain connections with discrete Morse theory and observations on the matrix structures.

Amongst the existing software Eirene~\cite{henselman2016matroid} seems to handle the representative computations efficiently and has thus been used for comparative purposes. It is based on fairly recent relationships between discrete Morse theory, matroid theory, and matrix factorizations that facilitate a more economical computation persistent homology and representative cycles than the standard algorithm~\cite{ELZ}. However, in contrast to our approach it does not employ the structure of cohomology.

The homology representatives computed through any of the mentioned approaches above are typically the initial representations obtained from persistent homology computations.
Further modifications and optimizations of these cycles have been treated in~\cite{Tamal1, Tamal2, Chao, Optimal1, Optimal2}.
In a different setting a nominally similar but essentially different problem has been considered in~\cite{Kozlov}. These approaches focus on modifications of representative cycles rather than their initial computation and thus complement our results.

\begin{figure}
  \includegraphics[width=.8\textwidth]{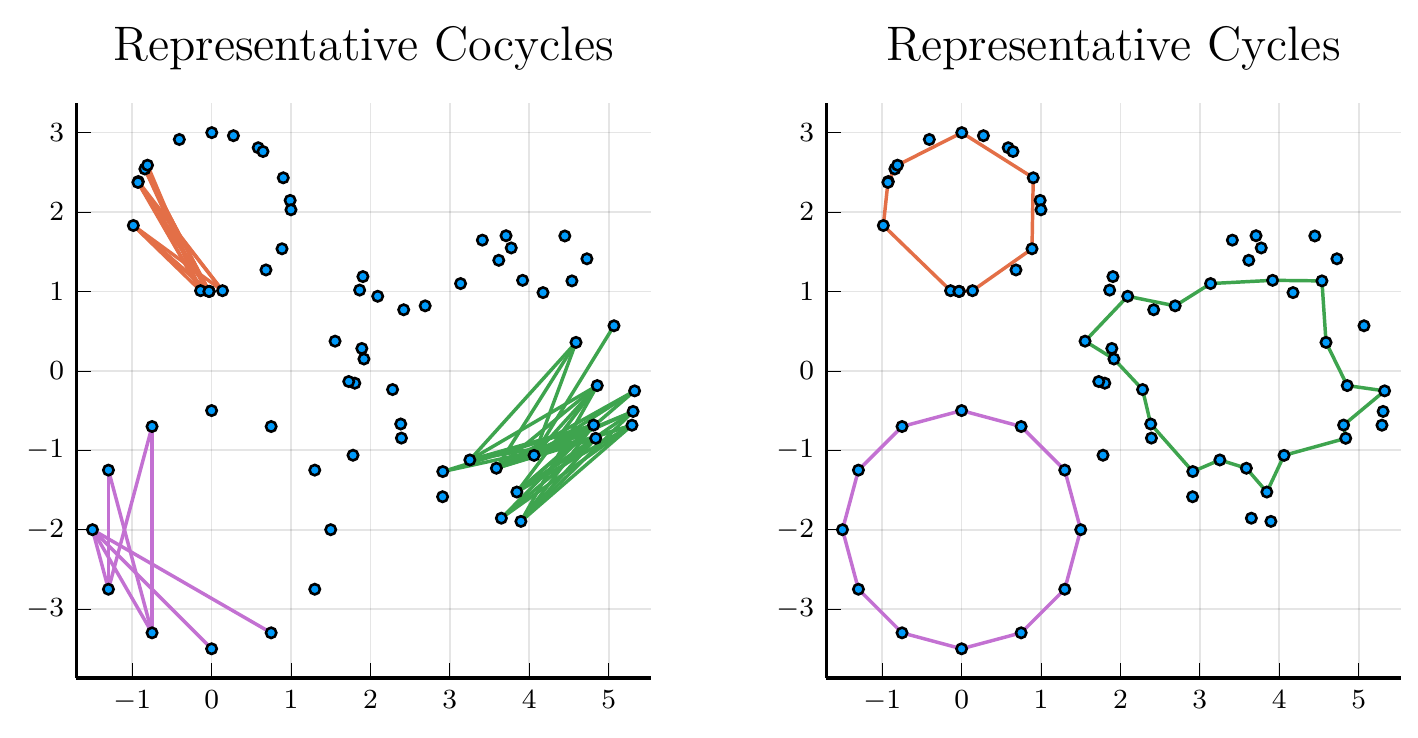}
  \caption{\label{fig:cycles-cocycles}
    Comparison of persistence cocycles and persistence cycles on a data set consisting of 20 points sampled randomly from a circle (red), 12 points sampled uniformly from a circle (purple), and 40 points sampled randomly from an annulus (green).}
  \end{figure}

\section{Theoretical background}

In this section we review theoretical background on persistent homology and cohomology, for further details on these topics see~\cite{EdelsHarer} and~\cite{de2011dualities}.

Persistent homology is a parameterized version of homology. Roughly speaking, it contains information on the holes in a space along with their sizes. In Subsection~\ref{SubsPH} we describe the structure of persistent homology arising from a filtration, i.e., a growing simplicial complex. The crucial observation here is that an addition of a single simplex either creates or terminates a hole in a simplicial complex. A hole of bounded lifespan corresponds to a specific persistence pair consisting of a birth and death simplex, while holes that never terminate have only an associated birth simplex. In Subsection~\ref{SubsRed} we recall the standard matrix reduction algorithm identifying the mentioned pairs of birth-death simplices. It is essentially a column reduction of the boundary matrix, proceeding from left to right. In Subsection~\ref{SubsRepresentative} we provide details on the standard way to extract homology representatives (representative cycles) from the reduced boundary matrix, which motivates phase 2 of our algorithm. Subsection~\ref{SubsCoh} provides a dual way to obtain the mentioned pairs of birth-death simplices by essentially reducing the coboundary matrix, i.e., the transpose of the boundary matrix. This approach turns out to be computationally more economical but does not allows us to extract homological representatives. However, we incorporate a version of it in phase 1 of our algorithm. We conclude by providing details on Rips complexes in Subsection~\ref{SubsRips}.
\subsection{Persistent homology}\label{SubsPH}

Throughout the paper we fix a field of coefficients $\FF$ for all mentioned vector spaces. Let $K$ be a finite simplicial complex.  A \textbf{filtration} of $K$ is a nested collection of subcomplexes:
$$
K_1 \leq K_2 \leq \ldots \leq K_m=K.
$$
A \textbf{filtration function} associated to a filtration is a function $\f$ assigning to each simplex $\sigma \in K$ the index $\f(\sigma)=\argmin_i \{\sigma \in K_i\}$. Throughout the paper we assume that $\f$ is injective, i.e., each $K_i$ is obtained from $K_{i-1}$ by an addition of a single simplex $\sigma_i$.

For $p\in\{0,1,\ldots\}$ we define the $p$-\textbf{chains} $C_p(K_i)$ as the vector space spanned by all oriented $p$-simplices of $K_i$. The $p$-\textbf{boundary map} is the linear map $\di_p\colon C_p(K)\to C_{p-1}(K)$ defined by
$$
\di_p\langle v_0, v_1, \ldots, v_p \rangle = \sum_{j=0}^{p} (-1)^j \langle v_0, \ldots, \hat v_j, \ldots,  v_p \rangle,
$$
where $\langle v_0, v_1, \ldots, v_p \rangle$ is a $p$-simplex in $K$, and $\langle v_0, \ldots, \hat v_j, \ldots,  v_p \rangle$ is the $(p-1)$-simplex in $K$ obtained by removing the vertex $v_j$ from the original simplex. We will also use the notation $\di_p$ to denote the corresponding boundary matrix, with the columns and rows being ordered in the ascending order of filtration function $\f$. Occasionally we will use $\di$ to denote the boundary map (and the corresponding square matrix) on all chains on all dimensions, with the ordering of rows and columns provided by $\f$.

\textbf{Homology} of $K_i$ in dimension $p$ is the quotient vector space
$$
H_p(K_i) = \frac{\Ker \di_p }{  \im \di_{p+1}}.
$$

For $i \leq j$ the inclusion $K_i \hookrightarrow K_{j}$ induces a linear map $H_p (K_i \hookrightarrow K_{j})$ on the level of homology.  \textbf{Persistent homology} in dimension $p$ is the collection of ranks of inclusion induced maps of a filtration on homology. In particular, it consists of \textbf{Betti numbers}
$$
\beta^p_{i,j} = \rank H_p (K_i \hookrightarrow K_{j}), \quad \forall 1 \leq i \leq j \leq m.
$$

Persistent homology is typically visualized and described by a persistence diagram, which contains the same information as persistent Betti numbers (see Fundamental Lemma of Persistent Homology in~\cite{EdelsHarer}). Persistence diagrams and the corresponding algorithm of Subsection~\ref{SubsRed} are motivated by the following observation.  For each $i$, the addition of and $n$-dimensional simplex $\sigma_i$ to $K_{i-1}$ either:
\begin{enumerate}
 \item creates an $n$-dimensional homology class, or
 \item destroys an $(n-1)$-dimensional class.
\end{enumerate}
Simplices of type (1) are called \textbf{birth simplices}. A simplex $\sigma_i$ is a birth simplex iff its boundary is a linear (with coefficients in $\FF$) combination of boundaries in $K_{i-1}$.

Simplices of type (2) are called \textbf{death simplices}. A simplex $\sigma_i$ is a death simplex iff its boundary is not a linear (with coefficients in $\FF$) combination of boundaries in $K_{i-1}$.

Each death $n$-simplex $\sigma_i$ is paired to a unique birth $(n-1)$-simplex $\sigma_j$ with $j < i$ to form a \textbf{persistence} (homology) \textbf{pair}. The addition of $\sigma_i$ destroys homology class created by $\sigma_j$. Birth simplices not contained in any persistence pair are called \textbf{essential} simplices.

A \textbf{persistence diagram} associated to a filtration is the collection of points $$\{(i,j) \mid (\sigma_i, \sigma_j) \text{ a persistence pair } \} \ \cup
$$
$$
  \{(i, \infty) \mid \sigma_i \text{ an essential simplex } \}.$$
  We next describe how to obtain persistence diagrams without the use of persistent Betti numbers.

\subsection{Reduction algorithm for persistent homology}\label{SubsRed}

The observations motivating the definitions of birth and death simplices led to the original reduction algorithm of~\cite{ELZ}, which returns persistence pairs and essential simplices. Let $\di$ be the full  boundary matrix of  $K$ with the columns and rows indexed by $\{1,2,\ldots, m\}$. Index $i$ represents simplex $\sigma_i$. Given a matrix $A$ whose rows and columns are indexed by $i\in \{1,2,\ldots, m\}$ define:
\begin{itemize}
\item  $\col_A(i)$ as the $i$-th column of $A$ as a vector.
\item  $\row_A(j)$ as the $j$-th row of $A$ as a vector.
\item  $\low_A(i)$ as the index of the lowest non-trivial entry in $\col_A(i)$ or $0$ if the column is trivial.
\end{itemize}

\begin{algorithm}[H]
 \caption{\label{Alg1}Column reduction algorithm for persistent homology.}
\KwData{boundary matrix $\di$}
 \KwResult{reduced boundary matrix}

\For{$i=1,2, \ldots, m$}
	{
	\For{$j=m-1,m-2, \ldots, 1$}
		{
		\If {$\low_\di(i)=\low_\di(j)\neq 0$}
			{
			$\lambda= \di(i,\low_\di(i))/ \di(j,\low_\di(j))$ \\
			$\col_\di(i) = \col_\di(i) - \lambda \col_\di(j)$
			}
		}
	}
\KwRet{$\di$}\\
\medskip
\end{algorithm}
\medskip

The reduction algorithm is essentially a column reduction process from left to right. The core idea is to verify whether the boundary of the added simplex $\sigma_i$ (i.e., $\col_\di(i)$) is homologically trivial in $K_{i-1}$ or not, by verifying whether it can be expressed as a linear combination of preceding columns.

Let $\di'$ denote the reduced boundary matrix as reduced by Algorithm~\ref{Alg1}. We can now \textbf{extract persistence diagrams} from the structure of $\di'$. If $\col_{\di'}(i)$ is not trivial then $(\sigma_{\low_{\di'}(i)}, \sigma_i)$ is a persistence pair. Simplices unpaired in this manner are essential simplices.

\subsection{Representatives}\label{SubsRepresentative}

The standard representative of a persistence pair $(\sigma_{\low_{\di'}(i)}, \sigma_i)$ is the vector represented by $\col_{\di'}(i)$. A representative is a chain whose homology class spans the homology that appeared at $\low_{\di'}(i)$ and died at $i$.

Given an essential simplex $\sigma_i$ Algorithm~\ref{Alg1} reduces  $\col_\di(i)$ to the trivial column, i.e., there exist $\lambda_j\in \FF$ such that $\col_\di(i) - \sum_{j=1}^{i-1} \lambda_j \col_\di(j)=0$. The standard representative corresponding to the essential simplex $\sigma_i$ is the cycle $\sigma_i - \sum_{j=1}^{i-1} \lambda_j \sigma_j$. This representative is a chain whose homology class appears at $\low_{\di'}(i)$ and never dies.

\subsection{Persistent cohomology}\label{SubsCoh}

While the definition of homology $H_p(K)$ is based on chains ($\FF$-combinations of simplices), the definition of cohomology $H^p(K)$ (see~\cite{Hatcher} for an introduction) is based on cochains ($\FF$-linear maps from the space of chains into $\FF$). It turns out that both invariants are isomorphic but that cohomology is contravariant in the sense that it reverses the direction of induced maps. The following result is a well known consequence of the universal coefficient theorem and contravariant functoriality of cohomology. In our setting it has first appeared in~\cite{de2011dualities}.

\begin{theorem}\label{ThmMain}
Let $K \hookrightarrow L$ be an inclusion of simplicial complexes. Then for each dimension $p$ there exists a commutative diagram
$$
\xymatrix
{H_p(K) \ar[d] \ar[r] & H_p(L) \ar[d] \\
H^p(K) & H^p(L) \ar[l]}
$$
with the vertical maps being isomorphisms.
 \end{theorem}

Persistent cohomology is constructed from a filtration (we still assume the filtration function is injective) in a similar way as persistent homology. We refrain from repeating the entire construction and instead point out the fundamental differences of the computational aspect:
\begin{enumerate}
\item Instead of $p$-chains we define the $p$-\textbf{cochains} $C^p(K_i)$, which are the collection of linear maps $C_p(K_i) \to \FF$.
\item The coboundary map $d_p\colon C^p(K)\to C^{p+1}(K)$ is defined as $d (\f) (\sigma)=\f(\di \sigma)$.
 \item The resulting standard coboundary matrix is the transpose of the boundary matrix. This results in a lower-triangular matrix and would require us to use the mentioned column reduction in the opposite direction to compute cohomology. For this reason we rather define our \textbf{coboundary matrix} $d$ to be the \textbf{anti-transpose} of $\di$ as in~\cite{de2011dualities}. In particular, $d$ is obtained from $\di^T$ by reversing the order of simplices labeling columns and rows.
 \item \textbf{Persistent cohomology} is computed by applying Algorithm~\ref{Alg1} to $d$ and extracting persistent cohomology pairs from the reduced matrix $d'$ as before.
 \item It turns out that $(\sigma_i,\sigma_j)$ is a persistence homology pair iff $(\sigma_j,\sigma_i)$ is a persistence cohomology pair. Essential simplices coincide in both cases.
 \item A cohomology representative of a persistence cohomology pair is a cochain, i.e., a linear map from the space of chains into $\FF$.
\end{enumerate}

For each $p\in \{0,1,\ldots\}$ we define $\di_p$ to be the $p$-dimensional boundary matrix and $d_p$ to be the $p$-dimensional coboundary matrix.
In particular, the columns of $\di_p$ are labeled by $p$-simplices and the rows are labeled by $(p-1)$-simplices.
Matrix $d_p$ is the anti-transpose of $\di_{p+1}$.  These matrices form the block structure of $\di$ and $d$ respectively.

\subsection{Rips complexes}\label{SubsRips}

Given a finite metric space $(X,d)$ and $r>0$ the \textbf{Rips complex} is defined as $\Rips(X,r)=\{\sigma \subseteq X \mid \diam(\sigma)\leq r\}$. The Rips filtration of $X$ is the collection of all Rips complexes of $X$ for all positive $r$. In order to obtain a filtration with an injective filtration function we order the simplices of a Rips filtration as $\sigma_1, \sigma_2, \ldots$ so that:
\begin{itemize}
 \item If $\diam(\sigma_i) < \diam(\sigma_j)$ then $i<j$, i.e., we first order simplices by diameter.
 \item If $\diam(\sigma_i) = \diam(\sigma_j)$ and $\dim(\sigma_i) < \dim(\sigma_j)$ then $i<j$, i.e., we then order simplices by dimension.
\item Simplices of the same diameter and dimension are ordered amongst themselves in an arbitrary order.
\end{itemize}

If we use the same such an ordering in the construction of $\di$ and $d$ (i.e., if $d$ is the anti-transpose of $\di$) then the resulting persistence $(\sigma_i, \sigma_j)$ pairs are the same by Theorem~\ref{ThmMain}. Choosing a different order above may change some of the persistence pairs. However, the persistence diagram of the Rips filtration, defined as the multiset of points
$$
\{(\diam(\sigma_i), \diam(\sigma_j)) \mid (\sigma_i, \sigma_j) \textrm{ a persistence pair} \} \ \cup
$$
$$
 \{(\diam(\sigma_i), \infty) \mid \sigma_i \textrm{ an essential simplex} \},
$$
remains unchanged. If $\diam(\sigma_i)= \diam(\sigma_j)$ for a persistence pair $ (\sigma_i, \sigma_j) $ we say the corresponding interval is trivial.

For practical reasons the scales of $r$ are often bounded from above. When that is not the case the Rips complex for large scales takes the form of the full simplex on all points, resulting in a single essential simplex (the first vertex in our chosen order).

\section{Algorithm}

In this section we present our algorithm to compute persistent homology representatives.
Our algorithm will make use of the discrepancy between the running times for persistent homology and cohomology computations. We restrict our treatment to Rips complexes although the same algorithm can be used for other constructions as well. Recall that $\di$ denotes a boundary matrix while $d$ denotes a coboundary matrix. In Subsection~\ref{SubsAlgo} we provide an overview of the algorithm, while the details of specific steps are discussed in Subsection~\ref{SubsDet}. We provide a small demonstrative example in Subsection~\ref{SubsEx}.

\subsection{The  algorithm}\label{SubsAlgo}

\textbf{Input}: Start with a finite metric space $X$.

\textbf{Preprocessing}:
Construct a fixed injective filtration function associated to the Rips filtration of $X$. We generate coboundary matrices $d_k$.

\textbf{The main part}:

\begin{enumerate}
\item Reduce each coboundary matrix $d_k$. By Theorem~\ref{ThmMain} we may extract homological death simplices (i.e., cohomological birth simplices) and essential simplices.
\item For each $k$ let $D_k$ be the submatrix of the homology boundary matrix $\di_k$ consisting of columns corresponding to homological death simplices and essential simplices. We keep the indices of simplices to label the columns.
\item Compute persistent homology representatives by reducing $D_k$ using Algorithm~\ref{Alg1}.
\end{enumerate}

\textbf{Output}: Return homology representatives from the reduced forms of matrices $D_k$.

\begin{remark}
The algorithm is stated so that it encompasses all dimensions. However, we can choose a fixed dimension $q$ and perform the algorithm only for the corresponding dimension. In particular, representatives of $q$-dimensional persistent homology classes are $q$-chains. We obtain them by first reducing the coboundary matrix $d_{q-1}$, and then generate and reduce the corresponding $D_{q}$.
\end{remark}

\subsection{Details on the  algorithm}\label{SubsDet}

Input, preprocessing and output parts have been explained in the previous section. Detailed comments are provided for the main part:

\begin{description}
 \item[Part (1)] In the first step, we can employ a variety of computational tricks that were developed to make persistent cohomology efficient, such as clearing~\cite{chen2011persistent}, skipping emergent pairs, and using an implicit coboundary matrix representation~\cite{bauer2019ripser}. These tricks naturally favor the structure of cohomology and result in a reduction of the cohomology matrix in a much shorter time than the corresponding boundary matrix. This part of computation is implemented in the package Ripserer.jl~\cite{vcufar2020ripserer} as the main workflow to compute persistent cohomology. \item[Part (2)] From the reduced matrices $d_k$ we deduce which columns of the boundary matrices $\di_k$ are required for extraction of homology representatives (formally, reduced $d_{k}$ determines the columns of $D_{k+1}$). These are the columns corresponding to homological death simplices (these correspond to cohomological birth simplices) and essential simplices. Matrices $D_k$ are theoretically obtained by restricting $\di_k$ to the corresponding columns. In practice we refrain from constructing the entire $\di_k$ by instead constructing the smaller $D_k$ directly from the mentioned simplices. The omitted columns correspond to the paired birth simplices: these would have been reduced to trivial columns in the reduction of the boundary matrices and would not contribute to the extraction of representatives.

Even though we know which rows of $\di_k$ will contain a pivot, matrices $D_k$ keep all rows from $\di_k$ as any of them might  appear in an expression of homology representatives. In a similar fashion we keep even the death simplices corresponding to trivial intervals of the persistence diagram (except for the ones mentioned in the next paragraph) of the Rips filtration as they may be needed in later reductions. Another thing we can take into account is that we know when the last non-trivial interval will die. This allows us to further truncate $D_k$ to only include death simplices up to the death time of the last non-trivial interval as representatives of trivial intervals are typically not of interest.

  \item[Part (3)]
When the scale $r$ in the Rips filtration is unrestricted, all but one simplex is paired and we have thus removed one less than half of the columns from $\di$. Furthermore, the reduction of columns to trivial columns typically takes longer than a partial reduction. As a result, the omissions of part (2) significantly speed up the column reduction performed in this part in comparison to application of Algorithm~\ref{Alg1} on the boundary matrix.

In practice the computation is usually restricted to low dimensions, i.e., $k << n=|X|$. In this case the improvement is even more pronounced. The number of columns of $d_k$ equals the number of $k$-simplices, which is ${{n}\choose{k+1}}=\O(n^{k+1})$. However, each death $k$-simplex is paired to a $(k-1)$-simplex. As the number of $(k-1)$-simplices is ${{n}\choose{k}}=\O(n^{k})$, there are at most ${{n}\choose{k}}=\O(n^{k})$ many death $k$-simplices. These label the columns of $D_k$ and thus number of columns of $D_k$ is at most ${{n}\choose{k}}=\O(n^{k})$. In summary, our algorithm reduces the number of columns from ${{n}\choose{k +1}}=\O(n^{k+1})$ to less than ${{n}\choose{k}}=\O(n^{k})$. See Proposition~\ref{PropQuant} for precise quantities.
\end{description}

The obtained persistence homology representatives are the same as the ones obtained by the direct reduction of $\di_k$ as the removed columns have no effect on the reduction of other columns.

\begin{proposition}\label{PropQuant}
 Let $X$ be a metric space consisting of $n$ points. Then the persistent homology of the full simplex on $X$ via any filtration function (for example, Rips filtration) has exactly one essential simplex, namely the first vertex in our chosen order. Furthermore, for each $k \geq 1$ the number of death simplices in dimension $k$, $\mathcal{S}_k$, equals
 $$
\mathcal S_k= {n \choose k}-  {n \choose k-1}+ {n \choose k-2}- \ldots \pm  {n \choose 0}.
 $$
\end{proposition}

\begin{proof}
 The first part is trivial as the full simplex on $X$ is contractible. This also means that the only essential simplex is a vertex. The formula of the second part follows by induction on $k$. For $k=1$, the number of non-essential birth vertices equals $n-1$ and coincides with the number of death edges, thus 
 $\mathcal{S}_1=n-1 = {n \choose 1} - {n \choose 0}$. Consequently, the number of  birth edges equals ${n \choose 2} - \mathcal{S}_1 = {n \choose 2} - (n-1)$ and coincides with the number of death triangles $\mathcal{S}_2 = {n \choose 2}-{n \choose 1} + {n \choose 0}$.

 We proceed by induction. Assuming the number of death simplices in dimension $k$ equals $\mathcal S_k$, we conclude the number of birth simplices in dimension $k$ is ${n \choose k+1}-\mathcal S_k$. The latter are all included in persistence pairs in dimensions $(k,k+1)$. Their number thus equals the number of death simplices in dimension $k+1$ yielding $\mathcal S_{k+1}=\binom{n}{k+1}-\mathcal S_k$.
\end{proof}

\subsection{Example}\label{SubsEx}

The following is a demonstrative example based on a filtration of the full simplex (tetrahedron) on four points $a,b,c,d$, in which the order of the $1$- and $2$-simplices is given by the order in the boundary matrix $\di_2$. We will be using coefficients in $\mathbb{Z}_2$.

We first demonstrate the standard column reduction algorithm. Using Algorithm~\ref{Alg1} we reduce matrix $\di_2$ column-wise from left to right to $\di'_2$. In our case, only the last column reduces to zero. The pivots (blue entries) describe persistence pairs $(bc, abc)$, $(bd, abd)$, and $(cd, acd)$. The corresponding representatives are encoded in the reduced columns. For example, the representative cycle of persistence pair  $(bc, abc)$ is $[ab]+[ac]+[bc]$. (While the said column was unchanged in this example, the columns in general change with the use of Algorithm~\ref{Alg1}).
$$
\di_2=\begin{pNiceArray}{cc|cc}[first-row,first-col]
&abc & abd & acd & bcd \\
ab &1&1&& \\
ac&1&&1&\\
\hline
ad&&1&1&\\
bc&1&&&1\\
\hline
cd&&&1&1\\
bd&&1&&1\\
\end{pNiceArray}
\qquad
\di'_2=\begin{pNiceArray}{cc|cc}[first-row,first-col]
&abc & abd & acd & bcd \\
ab &1&1&& \\
ac&1&&1&\\
\hline
ad&&1&1&\\
bc&\ii{3}&&&\\
\hline
cd&&&\ii{1}&\\
bd&& \ii{2}&&\\
\end{pNiceArray}
$$

We now demonstrate our algorithm on the same example. We first (implicitly) form the coboundary matrix $d_1$ as the anti-transpose of $\di_2$, and use the matrix reduction implemented in Ripserer.jl to reduce it to $d'_1$. The same result would be obtained using Algorithm~\ref{Alg1}. From $d'_1$ we extract the three cohomology birth (equivalently, homology death) simplices $abc, abd$, and $acd$, and generate (again, implicitly) the boundary submatrix $D_2$, which is the restriction of $\di_2$ to the columns of the mentioned homology death simplices. We then use Algorithm~\ref{Alg1} on $D_2$ to obtain $D'_2$, and extract homology representatives of persistence pairs as above.
$$
d_1=\begin{pNiceArray}{cc|cc|cc}[first-row,first-col]
&bd &  cd& bc & ad & ac & ab \\
bcd&1&1&1&&&\\
acd&&1&&1&1&\\
\hline
abd&1&&&1&&1\\
abc&&&1&&1&1\\
\end{pNiceArray}
\quad
d'_1=
\begin{pNiceArray}{cc|cc|cc}[first-row,first-col]
&bd &  cd& bc & ad & ac & ab \\
bcd&1&1&1&&&\\
acd&&\ii{6}&&&&\\
\hline
abd&\ii{5}&&&&&\\
abc&&&\ii{4}&&&\\
\end{pNiceArray}
$$
$$
D_2=\begin{pNiceArray}{cc|c}[first-row,first-col]
&abc & abd & acd\\
ab &1&1&\\
ac&1&&1\\
\hline
ad&&1&1\\
bc&1&&\\
\hline
cd&&&1\\
bd&&1&\\
\end{pNiceArray}
\qquad
D'_2=\begin{pNiceArray}{cc|c}[first-row,first-col]
&abc & abd & acd  \\
ab &1&1& \\
ac&1&&1\\
\hline
ad&&1&1\\
bc&\ii{7}&&\\
\hline
cd&&&\ii{8}\\
bd&& \ii{9}&\\
\end{pNiceArray}
$$
Note that $D'_2$ consists of all non-trivial columns of $\di'_2$, i.e., of all the columns that carry the information on the representative cycles. By restricting $\di_2$ to $D_2$ we have thus refrained from reducing the last column to zero. While in this demonstrative example our algorithm takes more time than the direct application of Algorithm~\ref{Alg1}, theoretical (Proposition~\ref{PropQuant}) and practical (see the experiments of the forthcoming section) evidence demonstrates that the improvement to the running time when computing lower-dimensional homology representatives on larger pointclouds can be significant.

\section{Experiments}

We implemented the algorithm presented in this paper in our Julia~\cite{bezanson2017julia} package Ripserer.jl~\cite{vcufar2020ripserer}. We compare the timings of our algorithm with the standard homology algorithm (also in Ripserer.jl), as well as Eirene.jl~\cite{henselman2016matroid}, which is a popular choice for computing representatives.

The benchmarks were performed using Julia v1.8.1, Ripserer v0.16.12, and Eirene v1.3.6 on a cluster equipped with Intel\textregistered Xeon\textregistered CPU E5-2680 v4 @ 2.40GHz. While we have not measured the memory consumption of the algorithms, it is worth noting that standard homology requires an order of magnitude more memory than either of the other two. We ran each benchmark 10 times and report the minimum running times. To ensure consistency, the benchmark for each data set was performed in a single continuous run.

We measured the runtimes for various data sets (Table~\ref{tab:eirene}, Figures~\ref{fig:speedups} and~\ref{fig:dim-scaling}) as well as increasing subsets of select datasets (Figure~\ref{fig:n-scaling}). We have also measured the sizes of boundary, coboundary, and involuted boundary matrices (Table~\ref{tab:matrix-size}).

Most of the data sets used in the experiments were taken from~\cite{otter2017roadmap} and~\cite{bauer2019ripser}. We have also added a new data set, gcycle. The data set included in the benchmarks are described below.

\begin{itemize}
\item \textbf{gcycle} is the distance matrix obtained by computing all shortest paths on a cycle graph. It has only one homological feature and is easy to scale up without changing the topological properties of the data.

\item \textbf{random16}~\cite{otter2017roadmap} is a set of 50 points sampled uniformly from $\mathbb{R}^{16}$.  It presents a good benchmark for scaling with homology dimension as it has no shortage of homological features even in higher dimensions. Random VR complexes were studied in~\cite{kahle2011random}.

\item \textbf{o3\_1024} and \textbf{o3\_4096}~\cite{bauer2019ripser} are sets of 1024 and 4096 orthogonal $3 \times 3$ matrices. Due to the larger number of points, these present a more computationally challenging benchmark.

\item \textbf{hiv}~\cite{otter2017roadmap} is derived from genomic sequences of the HIV virus. The authors of~\cite{otter2017roadmap} constructed a finite metric space using the independent and concatenated sequences of the three largest genes found in the genome. They compute Hamming distances between 1088 of the genes. The sequences were taken from the Los Alamos National Laboratory database. These sequences were also studied using PH in~\cite{chan2013topology}.

\item \textbf{celegans}~\cite{otter2017roadmap} is a weighted, undirected graph of the neuronal network of \textit{Caenorhabditis elegans} converted to a distance matrix. In this network, the nodes represent neurons, while edges represent the synapses or gap junctions. The network was studied using PH in~\cite{petri2013topological}.

\item \textbf{dragon}~\cite{otter2017roadmap} is a set of 1000 points sampled uniformly from the 3-dimensional scan of the Stanford Dragon.

\end{itemize}

\begin{figure}
  \includegraphics[width=.8\textwidth]{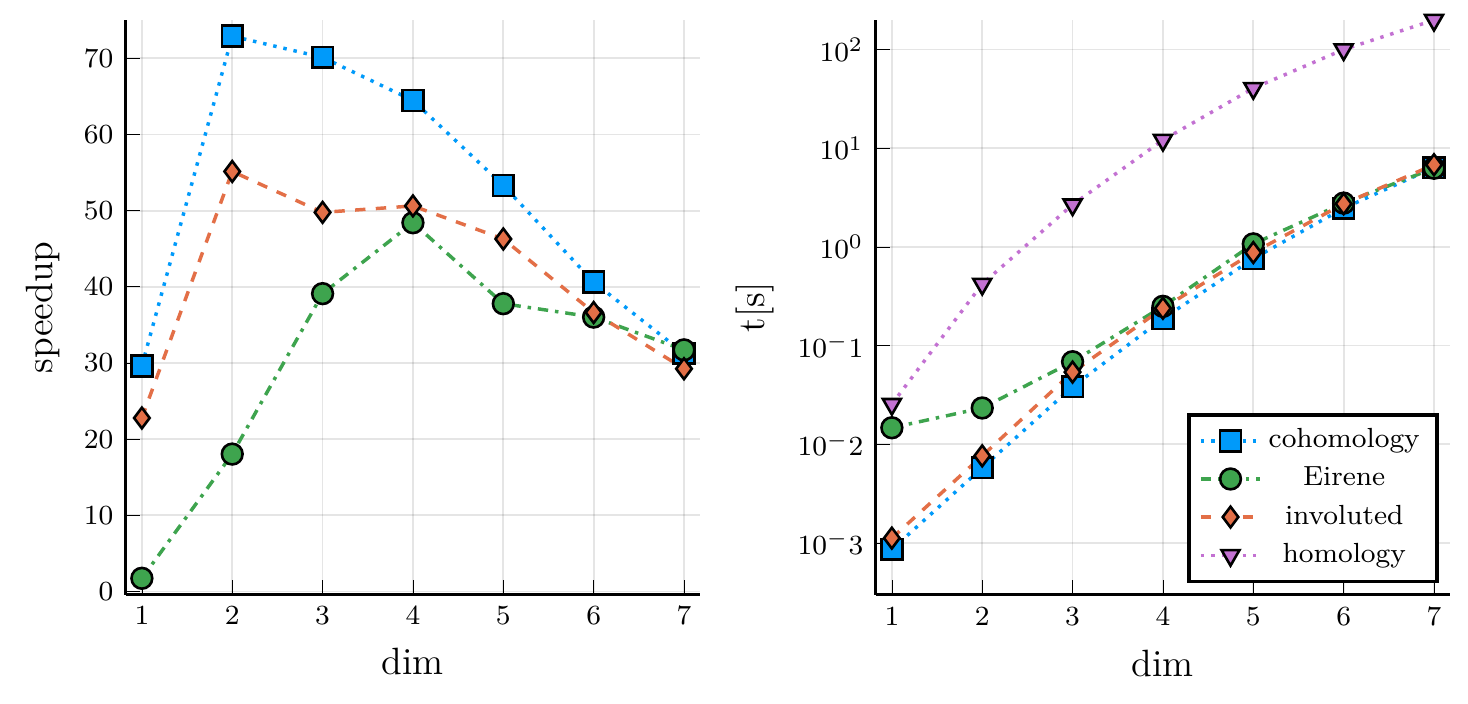}
  \caption{\label{fig:dim-scaling}Timings of our code and Eirene with increasing maximum homology dimension on a data set of 50 random points in $\mathbb{R}^{16}$. The left pane shows relative speedups compared to the homology computation, the right panes show the elapsed time in seconds in a logarithmic scale.}
\end{figure}

\begin{figure}
  \includegraphics[width=.8\textwidth]{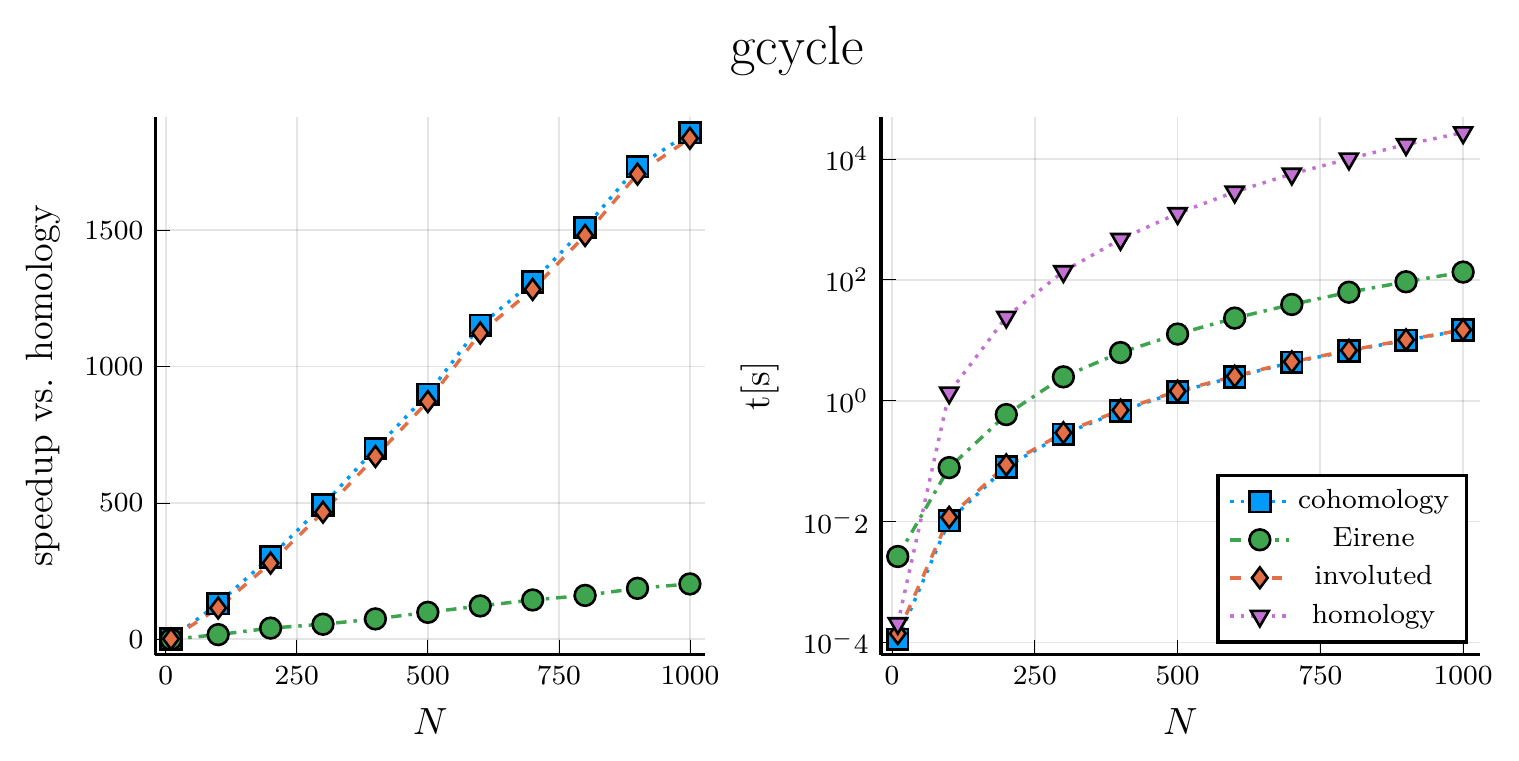}
  \includegraphics[width=.8\textwidth]{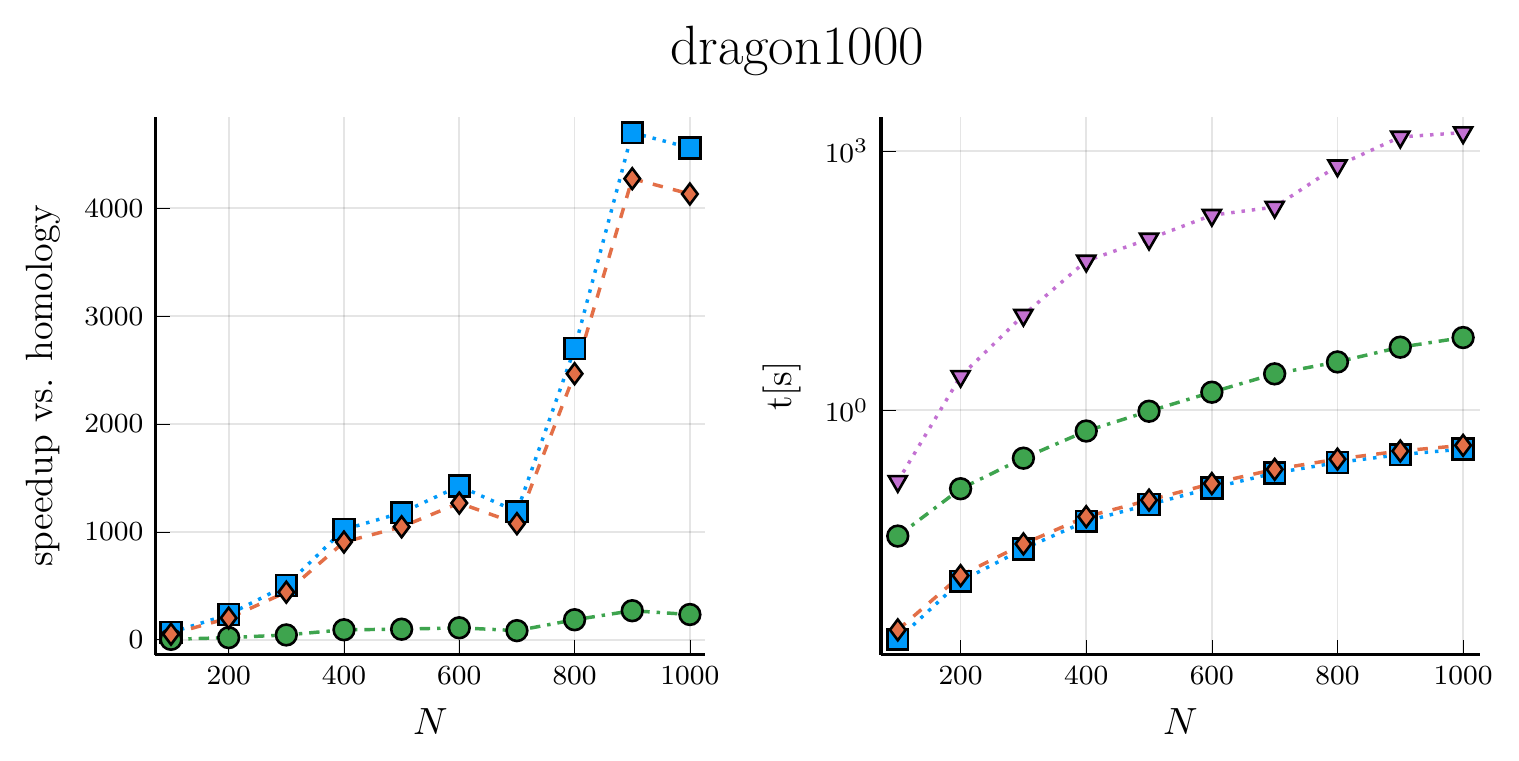}
  \includegraphics[width=.8\textwidth]{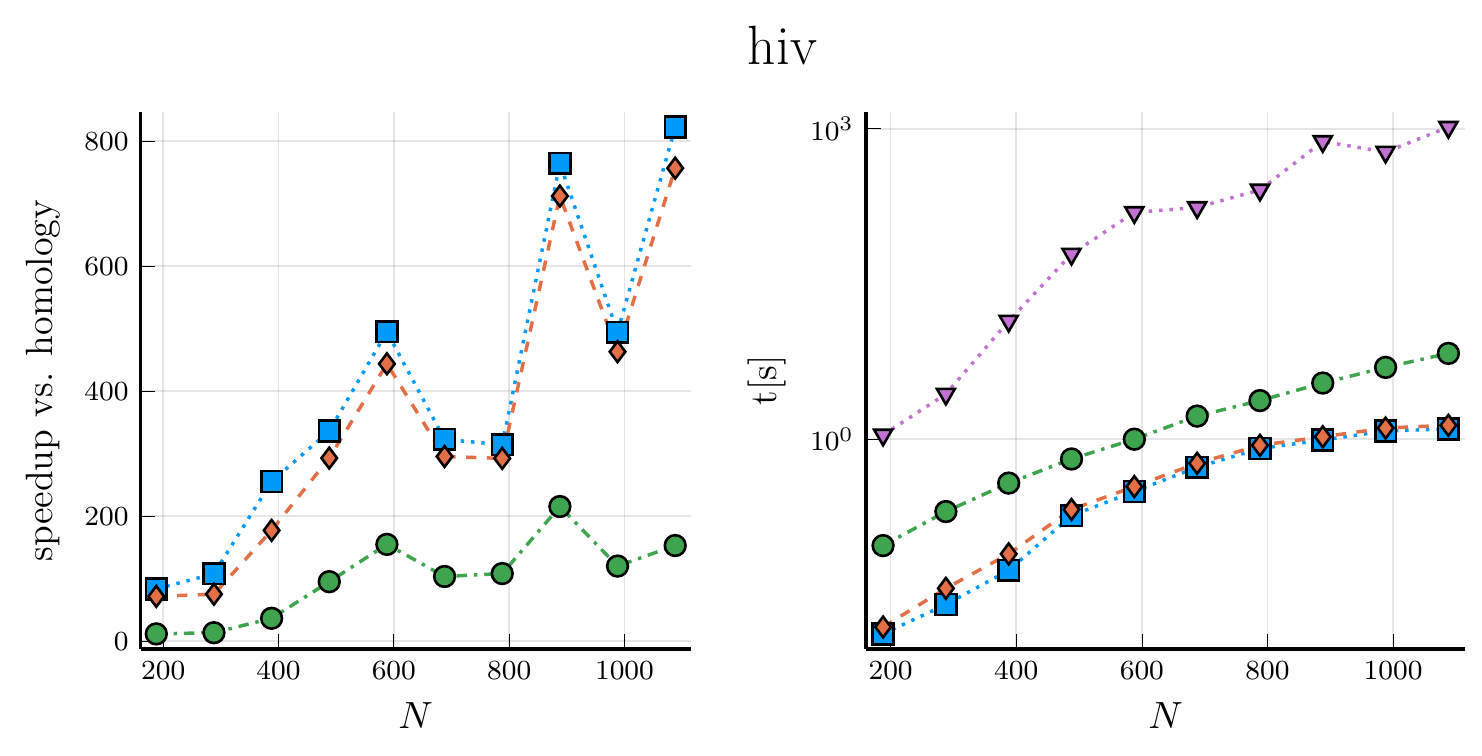}

  \caption{\label{fig:n-scaling}Timings of computing one-dimensional persistent homology with our code and Eirene on increasing number of points $N$ in three data sets. The left panes show relative speedups compared to the homology computation, the right panes show the elapsed time in seconds in a logarithmic scale. The maximal computeds dimension of persistent homology are noted in Table~\ref{tab:eirene}. The inconsistent improvements in the \textbf{dragon1000} and \textbf{hiv} can be explained by the fact that adding points to a complex data set may drastically change its persistent homological features. }
\end{figure}

\begin{figure}
  \includegraphics[width=.8\textwidth]{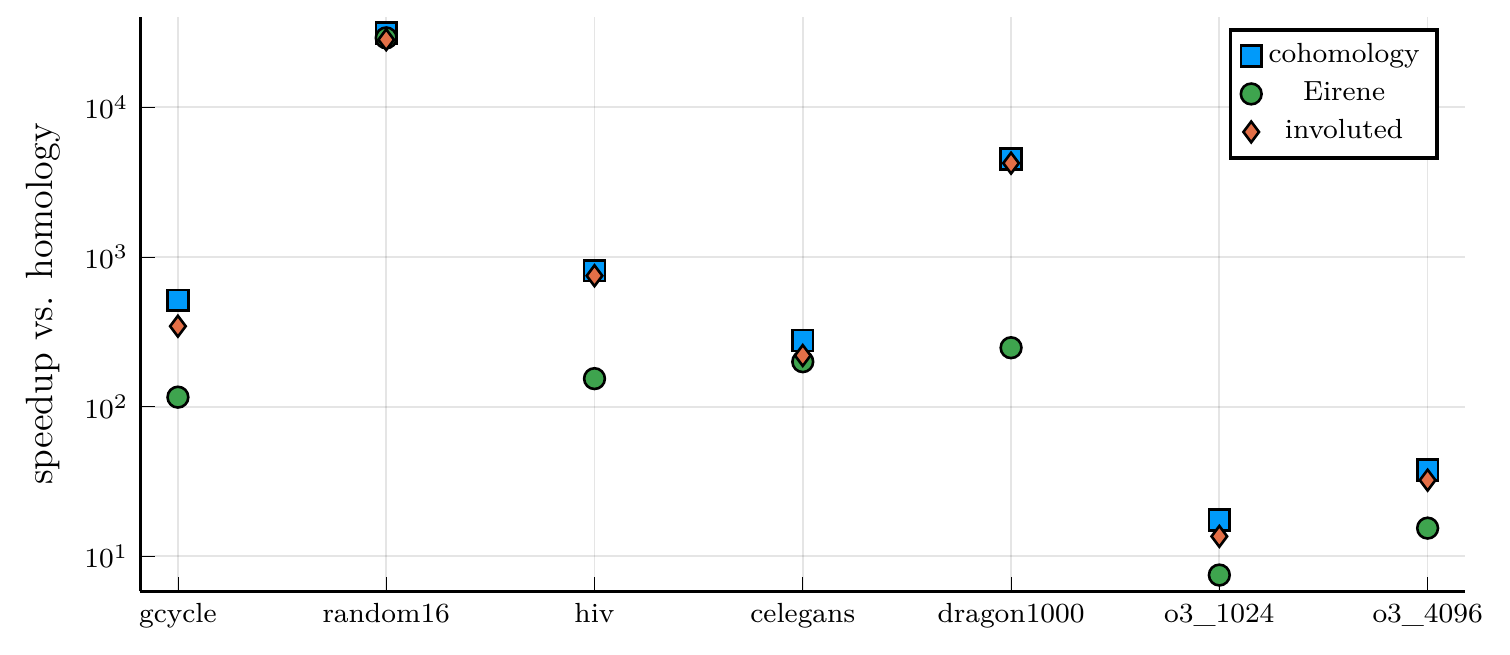}
  \caption{\label{fig:speedups}Speedups compared to homology computations for various data sets.}
\end{figure}

\begin{table}
  \begin{tabular}{c|c c c c c c}
    name & $N$ & $\dim$ & $r_{\max}$ & $m_d$ & $m_\partial$ & $m_D$ \\
    \hline\hline
    gcycle     &  100 & 3 & 40 &   106,021 &  10,602,020 &    106,021 \\
    random16   &   50 & 3 &    &   54,026  &     284,157 &     22,394 \\
    hiv        & 1088 & 1 &    &   517,622 & 155,009,693 &    165,623 \\
    celegans   &  297 & 2 &    & 3,735,740 & 256,663,737 &    918,390 \\
    dragon1000 & 1000 & 1 &    &   317,442 &  56,110,140 &     21,275 \\

  \end{tabular}
  \caption{\label{tab:matrix-size}Comparison of matrix sizes. $N$ is the number of points in the data set, $\dim$ the maximum homology dimension computed, and $r_{\max}$ the threshold applied. $m_d$, $m_\partial$, and $m_D$ are the number of columns of $d$, $\partial$, and $D$.}
\end{table}

\begin{table}
  \begin{tabular}{c|c c c c c c c}
    name & $N$ & $\dim$ & $r_{\max}$ & $t_h$ & $t_c$ & $t_i$ & $t_e$ \\
    \hline\hline
    gcycle     & 100  & 3 &  40 &  11 min & 1.262s  & 1.882s  & 5.606s \\
    random16   & 50   & 7 &     &    53 h & 6.079s  & 6.749s  & 6.564s \\
    hiv        & 1088 & 1 &     &  17 min & 1.272s  & 1.375s  & 6.696s \\
    celegans   & 297  & 2 &     &  12 min & 2.652s  & 3.336s  & 3.669s \\
    dragon1000 & 1000 & 1 &     &  28 min & 372.2ms & 396.4ms & 6.775s \\
    o3\_1024   & 1024 & 3 & 1.8 & 55.01 s & 3.150s  & 4.047s  & 7.331s \\
    o3\_4096   & 4096 & 3 & 1.4 & 1h 9min & 109.9s  & 128.7s  & 269.3s \\
  \end{tabular}
  \caption{\label{tab:eirene}Comparison of computation times between homology, cohomology, involuted homology, and Eirene. $N$ is the number of points in the data set, $\dim$ the maximum homology dimension computed, and $r_{\max}$ the threshold applied. $t_h$, $t_c$, $t_i$, and $t_e$ are the timings for homology, cohomology, involuted homology, and Eirene, respectively.}
\end{table}

\section{Analysis}

\subsection{Comparisons for coboundary, boundary and involuted boundary matrix reductions.}

The matrix sizes of $d$ (Table~\ref{tab:matrix-size}) are, as expected, lower than those of the homology reductions by orders of magnitude. On the other hand, they are larger than those of the reduction of $D$, sometimes by an order of magnitude. Thus for all our experiments, the involuted computation of homology of representatives is orders of magnitude faster than the direct reduction though homology matrix. While the involuted part of the algorithm (the reduction of $D$) naturally adds running time in addition to the cohomology reduction (the reduction of $d$ only), the total addition is mostly small (Figures~\ref{fig:dim-scaling},~\ref{fig:n-scaling}, and~\ref{fig:speedups}, Table~\ref{tab:eirene}) when compared to the cohomology reduction alone. As a result it appears that whenever cohomological reduction is feasible, so is the involuted homology reduction.

\subsection{Comparison of computation times between cohomology, involuted homology, and Eirene.} Our results suggests that involuted homology computation compares favorably (Figures~\ref{fig:dim-scaling},~\ref{fig:n-scaling}, and~\ref{fig:speedups}, Table~\ref{tab:eirene}). The only case when Eirene performed slightly better was the case of random points in $\mathbb{R}^{16}$, indicating that Eirene might be well suited to some types of data. On other datasets our approach performed much better, with speedups up to 17 times faster compared to Eirene's running time.

\section{Conclusions}

We have demonstrated the feasibility of the involuted persistent homology computations to obtain the persistent homology representatives. The running time of our algorithm is comparable to that of cohomology reduction (which does not yield homology representatives), and is orders of magnitude smaller than the standard homology reduction (the standard method to obtain the homology representatives). When compared to Eirene, our algorithm generally computes representatives faster (except for high-dimensional representatives of random points) with the improvement sometimes being of the orders of magnitude.

In summary, when computing persistent cohomology, it does not take much to compute the representative homology cycles as well. Furthermore, it appears that our approach is, in most cases, the most efficient way to obtain homology representatives.

\bibliographystyle{plain}
\bibliography{article}

\end{document}